\documentclass[a4paper,dvipdfm,11pt]{amsart}
\usepackage{amsmath}
\usepackage{amssymb}
\usepackage{enumerate}
\usepackage{amscd}
\usepackage{graphics}
\usepackage{latexsym}
\usepackage{verbatim} 
\usepackage{url}
\usepackage{bbm}

\theoremstyle{plain}
\newtheorem{thm}{{\bf Theorem}}[section]
\newtheorem{cor}[thm]{{\bf  Corollary}}
\newtheorem{prop}[thm]{{\bf Proposition}}

\newtheorem{claim}[thm]{{\bf Claim}}

\theoremstyle{definition}
\newtheorem{define}[thm]{{\bf Definition}}
\newtheorem{note}[thm]{{\bf Note}}

\newtheorem{question}[thm]{{\bf Question}}

\newcommand{\p}{\mathcal{P}}

\newcommand{\seq}[1]{\langle {#1} \rangle}

\newcommand{\om}{\omega}

\newcommand{\bbP}{\mathbb{P}}

\newcommand{\calL}{\mathcal{L}}

\newcommand{\ZF}{\mathsf{ZF}}

\newcommand{\AC}{\mathsf{AC}}

\title[Variants of \L o\'s's Theorem]{Variants of \L o\'s's Theorem}
\author[T. Usuba]{Toshimichi Usuba}
\address[T. Usuba]
{Faculty of Science and Engineering,
Waseda University, 
Okubo 3-4-1, Shinjyuku, Tokyo, 169-8555 Japan}
\email{usuba@waseda.jp}
\begin{document}

\begin{abstract}
We study \L o\'s's theorem in a choiceless context.
We introduce some variants of \L o\'s's theorem.
These variants seem weaker than \L o\'s's theorem, but
we prove that these are equivalent to \L o\'s's theorem.
\end{abstract}
\maketitle

\section{Introduction}
\emph{\L o\'s's theorem} is a fundamental theorem of
ultrapowers and ultraproducts.
First we recall it. 
Throughout this paper, a \emph{filter} over a set means a proper filter,
and a \emph{structure} does a first order structure with some language.
Let $U$ be an ultrafilter over a set $I$, and
$\{M_i \mid i\in I\}$ an indexed family of structures with language $\calL$.
For $f \in \prod_{i \in I} M_i$, let $[f]$ denote the 
equivalence class of $f$ modulo $U$.
Let $\prod_{i \in } M_i/U$ denote the ultraproduct of the family $\{M_i \mid i \in I\}$ by $U$.
If every $M_i$ is the same to the structure $M$,
then the ultraproduct $\prod_{i \in I} M_i/U$ is denoted by ${}^I M/U$, it is the ultrapower of $M$ by $U$.
\begin{thm}[\L o\'s's  fundamental theorem of ultraproducts \cite{L}]
For every $\calL$-formula $\varphi(v_0,\dotsc, v_n)$ and
$f_0,\dotsc, f_n \in \prod_{i \in I} M_i$,
\begin{align*}
\prod_{i \in I} M_i /U & \models \varphi([f_0],\dotsc, [f_n])\\
&\iff \{ i \in I \mid M_i \models \varphi(f_0(i),\dotsc, f_n(i))\} \in U.
\end{align*}
\end{thm}

In \cite{U}, we proved that some variants of \L o\'s's theorem 
and weak choice principles are equivalent to \L o\'s's theorem in $\ZF$:

\begin{thm}[\cite{U}]
In $\ZF$, the following are equivalent:
\begin{enumerate}
\item \L o\'s's theoreom, that is, 
for every family $\{M_i \mid i \in I\}$ of structures with same language and 
ultrafilter $U$ over $I$, if $\prod_{i \in I} M_i \neq \emptyset$ then
\L o\'s's fundamental theorem holds for the ultraproduct $\prod_{ i \in I} M_i/U$.
\item For every structure $M$ and ultrafilter $U$ over $I$, the ultrapower ${}^I M/U$ is  elementarily equivalent to $M$.
\item For every structure $M$ and ultrafilter $U$ over $I$, \L o\'s's fundamental theorem holds for the ultrapower ${}^I M/U$.
\item $U$-$\mathsf{AC}_I$ holds (Karagila-Yuan \cite{KY}) for every ultrafilter $U$ over $I$.
Where, for an ultrafilter $U$ over $I$, $U$-$\mathsf{AC}_I$ is the assertion 
that for every indexed family $\{A_i \mid i \in I\}$ of non-empty sets,
there is a function $f$ on $I$ such that $\{i \in I \mid f(i) \in A_i\} \in U$.
\end{enumerate}
\end{thm}
The \emph{ultrapower embedding} $j:M \to {}^I M/U$ is the map
defined by $j(x)=[c_x]$ for $x \in M$,
where $c_x:I \to M$  is the constant map with value $x$.
Under \L o\'s's theorem, the ultrapower embedding is an elementary embedding.
\begin{thm}[\cite{U}, (1) $\iff$ (4) is Bilinsky (\cite{EB})]\label{B}
 In $\ZF$, the following are equivalent:
\begin{enumerate}
\item \L o\'s's theorem.
\item For every structure $M$ and ultrafilter $U$ over $I$, 
the ultrapower embedding $j:M \to {}^IM /U$ is an elementary embedding.
\item For every structure $M$ and ultrafilter $U$ over $I$, 
if the ultrapower embedding $j:M \to {}^I M /U $ is
an elementary embedding, 
then \L o\'s's fundamental theorem holds for ${}^I M/U$.
\item For every structure $M$ and ultrafilter $U$ over $I$, if 
${}^I M/U$ is elementarily equivalent to $M$,
then the ultrapower embedding $j:M \to {}^IM /U$ is an elementary embedding.
\end{enumerate}
\end{thm}

In this paper, we study more variants of \L o\'s's theorem.
Variants which we will consider seem weaker than \L o\'s's theoreom,
but we prove that these are in fact equivalent to \L o\'s's theoreom in $\ZF$.

First we prove that $U$-$\AC_I$ can be characterized by
bounded forms of \L o\'s's fundamental theorem:
\begin{thm}\label{thm1}
In $\ZF$,
let $U$ be an ultrafilter over a set $I$.
Then the following are equivalent:
\begin{enumerate}
\item $U$-$\mathsf{AC}_I$ holds.
\item For every structure $M$,
the ultrapower ${}^I M/U$ is $\Sigma_2$-elementarily equivalent to $M$.
\item For every family $\{M_i \mid i \in I\}$ of pairwise isomorphic structures with same language,
if $\prod_{i \in I} M_i \neq \emptyset$ then
$\prod_{i \in I} M_i/U$ is $\Sigma_1$-elementarily equivalent to $M_i$ for every $i \in I$.
\item For every two structures $M$ and $N$ with same language,
if there are elementary embeddings $j:M \to N$ and $j':N \to M$,
then ${}^I M/U$ and ${}^I N/U$ are $\Sigma_2$-elementarily equivalent.
\item For every two families $\{M_i \mid i \in I\}$, $\{N_i \mid i \in I\}$ of structures with same language,
if each $M_i$ is isomorphic to 
$N_i$ and both 
$\prod_{i \in I} M_i$ and 
$\prod_{i \in I} N_i$ are non-empty,
then $\prod_{i \in I} M_i/U$ is $\Sigma_1$-elementarily 
equivalent to 
$\prod_{i \in I} N_i/U$.
\item For every two families $\{M_i \mid i \in I\}$,
$\{N_i \mid i \in I\}$ of structures  with  same language,
if each $M_i$ is isomorphic to 
$N_i$ and both 
$\prod_{i \in I} M_i$ and 
$\prod_{i \in I} N_i$ are non-empty,
then $\prod_{i \in I} M_i/U$ is isomorphic to 
$\prod_{i \in I} N_i/U$.
\end{enumerate}
\end{thm}

It is known that 
$U$-$\AC_I$ holds if, and only if,
for every family $\{M_i \mid i \in I\}$ of structures with same language,
\L o\'s's fundamental theorem holds for the ultraproduct $\prod_{i \in I} M_i/U$ (e.g., see \cite{U}).
As an immediate consequence, we have:
\begin{thm}\label{thm2}
In $\ZF$, the following are equivalent:
\begin{enumerate}
\item \L o\'s's theorem.
\item For every structure $M$ and  ultrafilter $U$ over $I$,
the ultrapower ${}^I M/U$ is $\Sigma_2$-elementarily equivalent to $M$.
\item For every family $\{M_i \mid i \in I\}$ of pairwise isomorphic structures with same language 
and  ultrafilter $U$ over $I$, if $\prod_{i \in I} M_i \neq \emptyset$ then
$\prod_{i \in I} M_i/U$ is $\Sigma_1$-elementarily equivalent to $M_i$ for every $i \in I$.
\item For every two structures $M$ and $N$ with same language and
ultrafilter $U$ over $I$,
if there are elementary embeddings $j:M \to N$ and $j':N \to M$,
then ${}^I M/U$ and ${}^I N/U$ are $\Sigma_2$-elementarily equivalent.
\item For every two families $\{M_i \mid i \in I\}$ 
and $\{N_i \mid i \in I\}$ of structures with same language
and ultrafilter $U$ over $I$, 
if each $M_i$ is isomorphic to 
$N_i$ and both 
$\prod_{i \in I} M_i$ and 
$\prod_{i \in I} N_i$ are non-empty,
then $\prod_{i \in I} M_i/U$ is $\Sigma_1$-elementarily 
equivalent to 
$\prod_{i \in I} N_i/U$.
\item For every two families $\{M_i \mid i \in I\}$ 
and $\{N_i \mid i \in I\}$ of structures  with  same language
and  ultrafilter $U$ over $I$, 
if each $M_i$ is isomorphic to 
$N_i$ and both 
$\prod_{i \in I} M_i$ and 
$\prod_{i \in I} N_i$ are non-empty,
then $\prod_{i \in I} M_i/U$ is isomorphic to 
$\prod_{i \in I} N_i/U$.
\end{enumerate}
\end{thm}

We also study a generic version of \L o\'s's theorem.
It is known that, in $\ZF$, \L o\'s's theorem is not provable,
and \L o\'s's theorem does not imply the Axiom of Choice $\AC$.
Unlike this, we show that a generic version of  \L o\'s's theorem
is equivalent to $\AC$ in $\ZF$.

\section{Proof of Theorem \ref{thm1}}
We start the proof of Theorem \ref{thm1}.
First we prove (1) $\iff$ (2) $\iff$ (3).
\begin{proof}[Proof of (1) $\iff$ (2)]
The direction (1) $\Rightarrow$ (2) is clear.

(2) $\Rightarrow$ (1). Take a family $\{A_i \mid i \in I\}$ of non-empty sets.
We shall find a function $f$ on $I$ such that
$\{i \in I \mid f(i) \in A_i\} \in U$.
Now we may assume that the family $\{A_i \mid i \in I\}$ is pairwise disjoint and
$I \cap \bigcup_{i \in I} A_i=\emptyset$.
Consider the structure $M=\seq{I \cup \bigcup_{i \in I} A_i; I, R}$,
where $I$ is identified with a unary predicate,
and $R$ is a binary relation defined by
$R(x,y) \iff x \in I$ and $y \in A_x$.
It is clear that the $\Pi_2$-sentence $\forall x(I(x) \to \exists y R(x,y))$ holds in $M$.
Let $M^*=\seq{M^*; I^*, R^*}$ be the ultrapower of $M$ by $U$,
where $I^*$ and $R^*$ are the relations corresponding  to $I$ and $R$.
By the assumption (2), $M^*$ is $\Sigma_2$-elementarily equivalent to $M$.
Hence $\forall x(I^*(x) \to \exists y R^*(x,y))$ holds in $M^*$.
Let $f:I \to I$ be the identity function.
We have $I^*([f])$, hence there is $[g] \in M^*$
such that $R^*([f], [g])$.
Then $\{i \in I \mid R(i,g(i))\} \in U$  by the 
construction of $M^*$,
so $\{i \in I \mid g(i) \in A_i\} \in U$.
\end{proof}

\begin{proof}[Proof of (1) $\iff$ (3)]
(1) $\Rightarrow$ (3) is clear.

(3) $\Rightarrow$ (1).
Take a family $\{A_i \mid i \in I\}$ of non-empty sets.
We may assume that $\{A_i \mid i \in I\}$ is a pairwise disjoint family.
Take a large limit ordinal $\theta$ with $A_i \in V_\theta$ for every $i \in I$.
Let $B_i=A_i \times V_\theta$.
Clearly there is an injection from $V_\theta$ into $B_i$, and
since $\theta$ is limit with $A_i \in V_\theta$,
we have $B_i \subseteq V_\theta$.
Thus for each $i, j \in I$ there is a bijection from $B_i$ onto $B_j$.
Let $S=\bigcup_{i \in I}B_i$, and
for $i \in I$, let $R_i=B_i$.
Let $M_i$ be the structure $\seq{S; R_i}$.
A $\Sigma_1$-sentence $\exists x R_i(x)$ holds in $M_i$.
\begin{claim}
For $i, j \in I$, $M_i$ is isomorphic to $M_j$.
\end{claim}
\begin{proof}
Fix a bijection $f:B_i \to B_j$.
Define $\pi:S \to S$ as follows.
For $x \in S$,
\[
\pi(x)=\begin{cases}
x  & \text{if $x \notin B_i \cup B_j$}.\\
f(x) & \text{if $x \in B_i$.}\\
f^{-1}(x) & \text{if $x \in B_j$.}
\end{cases}
\]
$\pi$ is a bijection from $S$ onto $S$ with $\pi``B_i=B_j$.
Hence $\pi$ is an isomorphism between $M_i$ and $M_j$.
\end{proof}
Let $M^*=\seq{M^*;R^*}$ be the ultraproduct $\prod_{i \in I} M_i/U$.
By the assumption (3),
$\exists x R^*(x)$ holds in $M^*$.
Pick $[f] \in M^*$ with
$R^*([f])$.
Then $X=\{i \in I \mid R_i(f(i))\} \in U$.
For $i \in X$, by the definition of $R_i$,
$f(i)$ is an element of $A_i \times V_\theta$.
Let $g(i)$ be the first coordinate of $f(i)$. 
We have $\{i \in I \mid g(i) \in A_i\} \in U$.
%
%
%
%
%
%
%
\end{proof}

Next we prove (1)$\iff$(4).

\begin{proof}[Proof of (1) $\iff$ (4)]
It is clear that (1) $\Rightarrow$ (4).
For (4) $\Rightarrow$ (1), 
take a family $\{A_i \mid i \in I\}$ of non-empty sets.
We may assume that $\{A_i \mid i \in I\}$ is a pairwise disjoint family
and $I \cap \bigcup_{i \in I} A_i=\emptyset$.
We shall construct two structures $M$ and $N$ such that
$N$ is an elementary substructure of $M$ and 
there is an elementary embedding $j:M \to N$.

First, let $A'_i=A_i\times \om$.
For $n<\om$ and $i \in I$, let $A_{i,n}'=A_{i}' \times \{i\} \times \{n\}$.
Fix a point $p \notin I \cup (I \times \om)
\cup \bigcup_{i \in I} A'_i \cup \bigcup_{n<\om, i \in I_n} A_{i,n}'$,
and let $B_{i,n}=A'_{i,n} \cup \{ \seq{p,i,n} \}$.
Let $M=I \cup (I \times \om) \cup \bigcup_{i \in I} A'_i \cup \bigcup_{i \in I, n<\om} B_{i,n}$ and 
$N=(I \times \om) \cup \bigcup_{i \in I, n<\om} B_{i,n}$.
We identify $M$ and $N$ as structures $\seq{M;I_0, R_0}$ and $\seq{N; I_0, R_1}$,
where $I_0=I \cup (I \times \om)$,
$I_1=I \times \om$ are unary predicates, and 
\begin{itemize}
\item $R_0(x,y) \iff x \in I$ and $y \in A_x'$, or $x=\seq{i,n} \in I \times \om$
and $y \in B_{i,n}$.
\item $R_1(x,y) \iff x =\seq{i,n}\in I \times \om$ and $y \in B_{i,n}$.
\end{itemize}
Note that $N$ is a substructure of $M$.
\begin{claim}
$N$ is an elementary substructure of $M$.
\end{claim}
\begin{proof}
Fix a formula $\varphi$,
and take $\seq{i_0,n_0},\dotsc,\seq{i_k,n_k} \in I \times \om$ and
$\seq{x_0, j_0, m_0},$ $\dotsc,\seq{x_l, j_l, m_l} \in \bigcup_{i \in I,n<\om} B_{i,n}$.
Fix a large natural number $h$ with $h>n_0,\dotsc, n_k, m_0,\dotsc, m_l$.
We work in some generic extension $V[G]$ in which $M$ and $N$ are countable.
Since $A_i'$ and $B_{i,n}$ are countably infinite,
we can find a bijection $\pi$ from $N$ onto $M$
such that:
\begin{enumerate}
\item $\pi(\seq{i,h})=i$ and $\pi(\seq{i,n})=\seq{i,n}$ for $i \in I$ and $n<h$.
\item $\pi(\seq{i,n})=\seq{i,n-1}$ for $i \in I$ and $n>h$.
\item $\pi \restriction \{\seq{i_0,n_0},\dotsc,\seq{i_k,n_k}\}$ is identity.
\item $\pi``B_{i,h}=A_i'$ for $i \in I$.
\item $\pi``B_{i,n}=B_{i,n}$ for $n<h$ and $i \in I$.
\item $\pi``B_{i,n}=B_{i,n-1}$ for $n >h$ and $i \in I$.
\item $\pi\restriction \{ \seq{x_0, j_0, m_0},$ $\dotsc,\seq{x_l, j_l, m_l} \}$ is identity.
\end{enumerate}
Then it is routine to check that $\pi$ is an isomorphism from $N$ onto $M$,
hence we have 
\begin{align*}
M & \models \varphi(\seq{i_0,n_0},\dotsc,\seq{i_k,n_k},
\seq{x_0, j_0, m_0}, \dotsc,\seq{x_l, j_l, m_l})\\
 \iff &
N \models \varphi(\seq{i_0,n_0},\dotsc,\seq{i_k,n_k},
\seq{x_0, j_0, m_0}, \dotsc,\seq{x_l, j_l, m_l}).
\end{align*}

\end{proof}
Next we define an embedding $j:M \to N$ as follows:
\begin{enumerate}
\item $j(i)=\seq{i,0}$ and $j(x)=\seq{x,i,0}$ for $i \in I$ and $x \in A_i'$.
\item $j(\seq{i,n})=\seq{i,n+1}$ and
$j(\seq{x, i,n})=\seq{x, i,n+1}$ for $ i \in I$ and $x \in A_{i,n}'$.
\end{enumerate}
\begin{claim}
$j:M \to N$ is an elementary embedding.
\end{claim}
\begin{proof}
Take a generic extension $V[G]$ in which $M$ and $N$ are countable.
In $V[G]$, since each $A_i'$ and $B_{i,0}$ are countably infinite,
for every finite set $C \subseteq \bigcup_{i \in I} A'_i$,
we can take a bijection $\pi$ from $\bigcup_{i \in I} A'_i$
onto $\bigcup_{i \in I} B_{i,0}$
such that $\pi \restriction C=j\restriction C$.
Then it is easy to see that $\pi \cup j\restriction( I \cup (I\times \om) \cup \bigcup_{i \in I, n<\om} B_{i,n})$
is an isomorphism, so $j$ is an elementary embedding as before.
\end{proof}

Let $M^*=\seq{M^*; I_0^*, R_0^*}$ and
$N^*=\seq{N^*; I_1^{*}, R_1^*}$ be ultrapowers of $M$ and $N$ by $U$ respectively.
\begin{claim}
$\forall x(I_1^{*}(x) \to \exists y R_1^*(x,y))$ holds in $N^*$,
\end{claim}
\begin{proof}
Take $[f] \in \prod_{i \in I}N_i/U$, and suppose
$I^{*}_1([f])$.
Then $X=\{i \in I \mid f(i) \in I_1\} \in U$.
For $i \in X$, we know that $f(i)$ is of the form 
$\seq{g(i), n}$ for some $g(i) \in I$ and $n<\om$.
Since $\seq{p,g(i), n}\in B_{g(i),n}$,
we have $h(i)=\seq{p,g(i),n} \in B_{g(i),n}$
and $R_1(f(i), h(i))$ holds.
Then clearly $R^*_1([f], [h])$ holds in $N^*$.
\end{proof}
By the assumption (4), 
$M^*$ is $\Sigma_2$-elementarily equivalent to $N^*$.
Hence 
$\forall x(I_0^{*}(x) \to \exists y R_0^*(x,y))$ holds in $M^*$.
Let $f \in {}^I M$ be the map $f(i)=i$.
We know $I^*_0([f])$ holds in $M^*$,
so there is some $[g] \in M^*$ with $R_0^*([f], [g])$ in $M^*$.
Then $X=\{i \in I \mid R_0(i, g(i))\} \in U$.
By the definition of $R_0$, $g(i)$ is an element of $A_i \times \om$ for $i\in X$.
Let $h(i) \in A_i$ be the first coordinate of $g(i)$.
Then $\{i \in I \mid h(i) \in A_i\}=X \in U$.
\end{proof}

%
Finally we show (1) $\iff$ (5) $\iff$ (6).
\begin{proof}[Proof of (1) $\iff$ (5) $\iff$ (6)]
For (1) $\Rightarrow$ (6),
take families $\{M_i \mid i \in I\}$ and $\{N_i \mid i \in I\}$ of structures
with  same language such that
each $M_i$ is isomorphic to $N_i$, and
$\prod_{i \in I} M_i, \prod_{i \in I} N_i \neq \emptyset$.
For $i \in I$, let $A_i$ be the set of all isomorphisms from $M_i$ onto $N_i$.
By the assumption (1), we can find a function $\sigma$ on $I$ such that 
$X=\{i \in I \mid \sigma(i) \in A_i\} \in U$.
Fix $h_0 \in \prod_{i \in I} N_i$.
Define $\pi:\prod_{i \in I} M_i \to \prod_{i \in I} N_i$ by
\[\pi(f)(i)=\begin{cases} \sigma(i)(f(i)) & \text{if $i \in X$,}\\
h_0(i) & \text{if $i \notin X$}.
\end{cases}\]
It is routine to check that
the assignment $[f] \mapsto [\pi(f)]$ is an isomorphism from $\prod_{i \in I} M_i/U$ onto $\prod_{i \in I} N_i/U$.
%


(6) $\Rightarrow$ (5) is clear.

For (5) $\Rightarrow$ (1), take a family $\{A_i \mid i \in I\}$ of non-empty sets.
For $i \in I$, let $A_{i}'=A_i \times \om$.
It is clear that $A_{i}'$ contains a countably infinite subset.
%

Fix two points $p, q \notin \bigcup_{ i\in I} A_{i}'$.
For $i \in I$, let $M_i$ and $N_i$ be the structures 
that $M_i=\seq{A_{i}' \cup\{p\}; p}$ and $N_i=\seq{A_{i}' \cup \{p,q\}; p}$,
where we identify $p$ as a constant symbol.
Since each $A_{i}'$ contains a countably infinite subset,
we can take a bijection from $A_i'$ onto $A_i' \cup \{q\}$.
This bijection induces an isomorphism from $M_i$ onto $N_i$.

We have $\prod_{i \in }M_i, \prod_{i \in I} N_i \neq \emptyset$,
hence $\prod_{i \in I} M_i/U=\seq{M^*; p^*}$
is $\Sigma_1$-elementarily equivalent to $\prod_{i \in I} N_i/U=\seq{N^*; p^{**}}$ by the assumption (5).
We know $\prod_{i \in I} N_i/U \models [c_q] \neq p^{**}$ (where $c_q$ is the constant function with value $q$),
so $\prod_{i \in I} N_i/U \models \exists x (x \neq p^{**})$ and
we have
$\prod_{i \in I} M_i/U \models \exists x (x \neq p^{*})$.
Pick $[g] \in \prod_{i \in I} M_i/U$
with $[g] \neq p^{*}$.
Then we have $X=\{i \in I \mid g(i) \neq p\} \in U$.
For each $i \in X$, since $g(i) \neq p$ we have $g(i) \in A_{i}'=A_i \times \om$.
Let $f(i)$ be the first coordinate of $g(i)$.
Then we have  $\{i \in I \mid f(i) \in A_i\}=X \in U$.
\end{proof}

%


\begin{note}
By the construction of ultrapowers,
in $\ZF$ it is easy to show that a structure $M$ is $\Sigma_1$-elementarily equivalent to
an ultrapower ${}^I M/U$.
In this sense, $\Sigma_2$-elementarity in (2) and (4) of Theorem \ref{thm1} is optimal.
\end{note}

Variants  of (3) and  (6) in Theorem \ref{thm1} are
still equivalent to $U$-$\AC_I$.
For sets $A$, $B$, let us say that $A$ and $B$ are \emph{equipotent}, denoted by $A \simeq B$,
if there is a bijection from $A$ onto $B$.
For an ultrafilter $U$ over $I$
and a family $\{A_i \mid i \in I\}$ of non-empty sets,
let $\prod_{i \in I} A_i/U$ be just the set $\{[f] \mid f \in \prod_{i \in I} A_i\}$,
where $[f]$ denotes the equivalence class of $f$ modulo $U$.

\begin{prop}
In $\ZF$, let $U$ be an ultrafilter over $I$.
Then the following are equivalent:
\begin{enumerate}
\item $U$-$\AC_I$ holds.
\item For every two families $\{A_i \mid i \in I\}$, $\{B_i \mid i \in I\}$ of non-empty sets,
if $A_i \simeq B_i$ for every $i \in I$ and
both $\prod_{i \in I} A_i, \prod_{i \in I} B_i$ are non-empty,
then $\prod_{ i \in I} A_i/U \simeq \prod_{i \in I} B_i/U$.
\item For every family $\{A_i \mid i \in I\}$ of pairwise equipotent sets,
if each $A_i$ has at least two elements and $\prod_{i \in I} A_i \neq \emptyset$,
then $\prod_{ i \in I} A_i/U$ has two elements.
\end{enumerate}
\end{prop}
\begin{proof}
(1) $\Rightarrow$ (2) is immediate from (6) in Theorem \ref{thm1}.
For (2) $\Rightarrow$ (1), we repeat the proof of (5) $\Rightarrow$ (1) in Theorem \ref{thm1}.
Take a family $\{A_i \mid i \in I\}$ of non-empty sets,
and let $A_i'=A_i\times \om$.
Fix $p, q \notin \bigcup_{i \in I} A_i'$.
Then for each $i \in I$, we have $A_i' \cup \{p\} \simeq A_i' \cup \{p,q\}$.
Moreover $\prod_{i \in I} (A_i' \cup \{p\})$ and $\prod_{i \in I} (A_i' \cup \{p,q\})$ are non-empty.
By the assumption,
there is a bijection $f: \prod_{i \in I} (A_i' \cup \{p\})/U \to \prod_{i \in I} (A_i' \cup \{p,q\})/U$.
Clearly $\prod_{i \in I} (A_i' \cup \{p,q\})/U$ has at least two elements,
hence so does $\prod_{i \in I} (A_i' \cup \{p\})/U$.
Pick two elements $[f], [g] \in \prod_{i \in I} (A_i' \cup \{p\})/U$.
Then $\{i \in I \mid f(i) \neq p\} \in U$ or
$\{i \in I \mid g(i) \neq p\} \in U$.
If $\{i \in I \mid f(i) \neq p\} \in U$, then $\{i \in I \mid f(i) \in A_i \times \om\} \in U$,
and we can take a function $f'$ with $\{ i\in I \mid f'(i) \in A_i\}$.
The case $\{i \in I \mid g(i) \neq p\} \in U$ is similar.

(1) $\Rightarrow$ (3). 
For a given family $\{A_i \mid i \in I\}$,
since $\prod_{i \in I} A_i \neq \emptyset$,
we can fix $f \in \prod_{i \in I} A_i$.
For each $i \in A_i$, since $A_i$ has two elements, we have
$A_i \setminus \{f(i)\} \neq \emptyset$.
By (1), we can take  a function $g$ on $I$ with
$\{i \in I \mid g(i) \in A_i \setminus \{f(i)\}\} \in U$.
Let $h$ be the function on $I$ defined by
$h(i)=g(i)$ if $g(i) \in A_i \setminus \{f(i)\}$,
and $h(i)=f(i)$ otherwise.
We have $[f] \neq [h]$, so
$\prod_{ i \in I} A_i/U$ has two elements.

(3) $\Rightarrow$ (1).
Let $\{A_i \mid i \in I\}$ be a family of non-empty sets.
In this case we repeat the proof of (3) $\Rightarrow$ (1) in Theorem \ref{thm1}.
Fix a large limit ordinal $\theta$ with
$A_i \in V_\theta$ for every $i \in I$.
We know $A_i \times V_\theta \simeq A_j \times V_\theta$ for every $i, j \in I$.
Fix a point $p \notin \bigcup_{ i \in I} (A_i \times V_\theta)$, and
let $B_i=(A_i \times V_\theta) \cup \{p\}$.
Each $B_i$ has at least two elements,
$B_i \simeq B_j$ for $i, j \in I$, and $\prod_{i \in I} B_i \neq\emptyset$,
hence $\prod_{i \in I} B_i/U$ has two elements by (3).
In particular there is $[f]  \in \prod_{i \in I} B_i/U$ with $[f] \neq [c_p]$.
Then $\{i \in I \mid f(i) \neq p\} \in U$,
hence if $f'(i)$ is the first coordinate of $f(i)$ then
$\{i \in I \mid f'(i) \in A_i\} \in U$.
\end{proof}

\begin{cor}
In $\ZF$, the following are equivalent:
\begin{enumerate}
\item \L o\'s's theorem.
\item For every two families $\{A_i \mid i \in I\}$, $\{B_i \mid i \in I\}$ of non-empty sets
and ultrafilter $U$ over $I$,
if $A_i \simeq B_i$ for every $i \in I$ and
both $\prod_{i \in I} A_i, \prod_{i \in I} B_i$ are non-empty,
then $\prod_{ i \in I} A_i/U \simeq \prod_{i \in I} B_i/U$.
\item For every family $\{A_i \mid i \in I\}$ of pairwise equipotent sets and
ultrafilter $U$ over $I$,
if each $A_i$ has at least two elements and $\prod_{i \in I} A_i \neq \emptyset$
then $\prod_{ i \in I} A_i/U$ has two elements.
\end{enumerate}
\end{cor}

\begin{question}
In the spirit of Theorem \ref{B},
the following statement would be considerable:
For every $M$ and $U$, if ${}^I M/U$ is $\Sigma_2$-elementarily equivalent to $M$,
then
${}^I M/U$ is elementarily equivalent to $M$.
In $\ZF$, is this statement equivalent to \L o\'s's theorem?
\end{question}

\section{Generic \L o\'s's theorem}
First we recall generic ultrapower and generic ultraproduct.

\begin{define}
For a filter $F$ over $I$,
let $F^+$ be the set 
$\{X \in \p(I) \mid X \cap Y \neq \emptyset$ for every $Y \in F\}$.
An element of $F^+$ is called an $F$-positive set.
Let $\bbP_F$ be the poset $F^+$
with 
the order defined by $X \le Y \iff X \setminus Y \notin F^+$.
\end{define}

If $G$ is a $(V, \bbP_F)$-generic filter, then
$G$ is a $V$-ultrafilter extending $F$,
 that is, the following hold:
\begin{enumerate}
\item $F \subseteq G$ and $\emptyset \notin G$.
\item For $X \in G$ and $Y \in \p(I)^V$, if $X \subseteq Y$ then $Y \in G$.
\item For $X, Y \in G$ we have $X \cap Y \in G$.
\item For every $X \in \p(I)^V$, either $X \in G$ or else $I \setminus X \in G$.
\end{enumerate}
For a family $\{M_i \mid i \in I\}\in V$ of structures,
we can define the equivalence relation on $(\prod_i M_i)^V$ modulo $G$
and the equivalence class $[f]$ for $f \in (\prod_{i \in I} M_i)^V$ as expected.
Hence we can construct the \emph{generic ultraproduct} $\prod_{i \in I} M_i/G$ in $V[G]$,
and, similarly,  for a structure $M \in V$, we can take the \emph{generic ultrapower} ${}^I M/G$ in $V[G]$.
Under $\mathsf{AC}$ in $V$,
we have the following generic version of \L o\'s's fundamental theorem:
For every  formula $\varphi(v_0,\dotsc, v_n)$ and $f_0,\dotsc, f_n  \in (\prod_{i \in I} M_i)^V$,
\begin{align*}
\prod_{i \in I} M_i /G & \models \varphi([f_0],\dotsc, [f_n])\\
& \iff \{i \in I \mid M_i \models \varphi(f_0(i),\dotsc, 
f_n(i) )\} \in G.
\end{align*}
In particular, under $\AC$ in $V$, the generic ultrapower ${}^I M/G$ is elementarily equivalent to $M$.


\begin{prop}\label{prop}
In $\ZF$, the following are equivalent:
\begin{enumerate}
\item The Axiom of Choice $\mathsf{AC}$.
\item For every structure $M$, filter $F$ over $I$, and 
$(V, \bbP_F)$-generic filter $G$,
the generic ultrapower ${}^I M/G$ is $\Sigma_2$-elementarily equivalent to $M$.
\item For every structure $M$, filter $F$ over $I$, and 
$(V, \bbP_F)$-generic filter $G$,
the generic ultrapower ${}^I M/G$ is elementarily equivalent to $M$.
\item For every family $\{M_i \mid i \in I\}$ of pairwise isomorphic structures with same language, filter $F$ over $I$, and 
$(V, \bbP_F)$-generic filter $G$, if $\prod_{i \in I} M_i \neq \emptyset$ then
the generic ultraproduct $\prod_{i \in I} M_i/G$ is $\Sigma_1$-elementarily equivalent to $M_i$
for every $i \in I$.
\end{enumerate}
\end{prop}
\begin{proof}
(1) $\Rightarrow$ (3), (4) are well-known, and (3) $\Rightarrow$ (2) is trivial.

(2) $\Rightarrow$ (1).
Our proof is based on Howard's one (\cite{H}).
Take an indexed family $\{A_i \mid i \in I\}$ of non-empty sets,
and suppose to the contrary that this family has no choice function.
Let $F$ be the set of all $X \subseteq I$ such that
$\{A_i \mid i \in I \setminus X\}$ has a choice function ($X=I$ is possible).
One can check that $F$ is a filter over $I$.
We define the structure $M$ as the following.
Let $S=I \cup \bigcup_{i \in I} A_i$,
and let $R \subseteq S^2$ be the binary relation defined by
$R(x,y) \iff x \in I$ and $y \in A_x$.
Let $M$ be the structure $\seq{S; I, R}$.
We know that the $\Pi_2$-sentence $\forall x(I(x) \to \exists yR(x,y))$ holds in $M$.

Take a $(V, \bbP_F)$-generic $G$,  and we work in $V[G]$.
Let $M^*=(M^*; I^*, R^*)$ be the generic ultrapower of $M$ by $G$.
By (2),
we have that $M^* \models \forall x(I^*(x) \to \exists yR^*(x,y))$.
If $\mathrm{id}$ is the identity map on  $I$,
we have $M^* \models I^*([\mathrm{id}])$
by the construction of the generic ultrapower.
Hence there is $f \in ({}^I M)^V$
such that
$M^*\models R^*([\mathrm{id}], [f])$.
Then $X=\{i \in I \mid R(i, f(i))\} =\{i \in I \mid f(i) \in A_i\} \in G$.
Because $f \in V$,
the family $\{A_i \mid i \in X\}$ has a choice function $f$ in $V$.
This means that $I \setminus X \in F$.
Since $F \subseteq G$, we have $I \setminus X \in G$,
but this contradicts $X \in G$.

(4) $\Rightarrow$ (1).
Take a family $\{A_i \mid i \in I\}$ of non-empty sets which
has no choice function.
Let $F$ be the filter as the above,
and let $\{M_i=(M_i; R_i) \mid i \in I\}$ the family of pairwise isomorphic structures
in the proof of  (3) $\Rightarrow$ (1) in Theorem \ref{thm1}.
Take a $(V, \bbP_F)$-generic $G$, and work in $V[G]$.
If $M^*=(M^*;R^*)$ is the generic ultraproduct $\prod_{i \in I} M_i/G$,
we have $\exists x R^*(x)$ holds in $M^*$,
hence there is $[f] \in M^*$ with $R^*([f])$ in $M^*$.
Then $f \in V$ and $\{i \in I \mid R_i(f(i)) \} \in G$,
and we can choose $g \in V$ with $\{i \in I \mid g(i) \in A_i\} \in G$.
Then  we can derive a contradiction as the above.
\end{proof}
\begin{note}
In $\ZF$, a structure $M$ is always $\Sigma_1$-elementarily equivalent to a generic ultrapower ${}^I M/G$.
Hence $\Sigma_2$-elementarity of (2) in Proposition \ref{prop} is optimal.
\end{note}
We can prove that generic versions of (4)--(6) in Theorem \ref{thm2} are equivalent to $\AC$,
and we omit the proof.


\begin{thebibliography}{100}
\bibitem{EB} E.~Bilinsky, private communication.
\bibitem{H} P.~E.~Howard, \emph{\L o\'s' theorem and the Boolean prime ideal theorem imply the axiom of choice},
Proc. Amer. Math. Soc. 49 (1975), 426--428.
\bibitem{KY} A.~Karagila, J.~Yuan, \emph{Critical embeddings}, preprint,
available at \url{https://arxiv.org/abs/2401.02951}
\bibitem{L} J.~\L o\'s, \emph{Quelques remarques, theoremes et problemes sur les classes definissables d'algebres},
Mathematical interpretation of formal systems, 98--113 (1955).

\bibitem{U} T.~Usuba, \emph{
A note on \L o\'s's theorem without the Axiom of Choice},
Bulletin Polish Acad. Sci. Math. 72 (2024), 17--44. 
\end{thebibliography}
\end{document}